\documentclass[10pt, a4paper]{article}
\usepackage{amsmath}
\usepackage{amsthm}
\usepackage{amssymb}
\usepackage{array}
\usepackage{multirow}
\usepackage{graphicx}
\usepackage{epsfig}
\usepackage{mathrsfs}
\usepackage{hhline}
\usepackage{longtable}
\usepackage{array}
\usepackage{enumerate}
\usepackage{xcolor}
\usepackage{float}
\usepackage{subfig}

\graphicspath{%
    {converted_graphics/}
    {/}
}

\newcolumntype{"}{@{\hskip\tabcolsep\vrule width 1pt\hskip\tabcolsep}}

\definecolor{vividviolet}{rgb}{0.62, 0.0, 1.0}

\newtheoremstyle{de}
  {10pt}          
  {10pt}  
  {\rm}  
  {}
  {\bf}  
  {. }    
  { }    
  {}     
\theoremstyle{de}

\newtheorem{example}{Example}[section]

\newtheoremstyle{theorem}
  {10pt}          
  {10pt}  
  {\it}  
  {}
  {\bf}  
  {. }    
  { }    
  {}     
\theoremstyle{theorem}

\numberwithin{equation}{section}

\newtheorem{theorem}{Theorem}[section]
\newtheorem{lemma}[theorem]{Lemma}

\newtheorem{corollary}[theorem]{Corollary}
\newtheorem{conjecture}{Conjecture}[section]
\numberwithin{equation}{section}

\begin{document}
\baselineskip18truept
\normalsize
\begin{center}
{\mathversion{bold}\Large \bf A note on local antimagic chromatic number of lexicographic product graphs }

\bigskip
Gee-Choon Lau$^{1}$, Wai-Chee Shiu$^{2}$, K. Premalatha$^{3}$, Ruixue Zhang$^{4,}\footnote{Corresponding author (ruixuezhang7@163.com)}$, M. Nalliah$^{5}$\\

\medskip

\emph{{$^a$}Faculty of Computer \& Mathematical Sciences,}\\
\emph{Universiti Teknologi MARA (Johor Branch, Segamat Campus),}\\
\emph{85000, Malaysia.}\\
\emph{geeclau@yahoo.com}\\

\medskip
\emph{{$^b$}Department of Mathematics,}\\
\emph{The Chinese University of Hong Kong,}\\
\emph{Shatin, Hong Kong, China.}\\
\emph{wcshiu@associate.hkbu.edu.hk}\\

\medskip
\emph{{$^c$}National Centre for Advanced Research in Discrete Mathematics,\\Kalasalingam Academy of Research and Education, Krishnan koil, India.}\\
\emph{premalatha.sep26@gmail.com}\\

\medskip
\emph{{$^d$} School of Mathematics and Statistics, \\Qingdao University, Qingdao, 266071, China.}\\
\emph{rx.zhang87@qdu.edu.cn}

\medskip
\emph{{$^e$}Department of Mathematics, School of Advanced Sciences, \\Vellore Institute of Technology, Vellore-632 014, India.} \\
\emph{nalliahklu@gmail.com}

\end{center}

 \begin{abstract}
 \noindent
Let $G = (V,E)$ be a connected simple graph. 
A bijection $f : E \to \{1,2,\ldots,|E|\}$ is called a local antimagic labeling of $G$ if $f^+(u) \ne f^+(v)$ holds for any two adjacent vertices $u$ and $v$, where $f^+(u) = \sum_{e\in E(u)} f(e)$ and $E(u)$ is the set of edges incident to $u$.
A graph $G$ is called local antimagic if $G$ admits at least a local antimagic labeling.
The local antimagic chromatic number, denoted $\chi_{la}(G)$, is the minimum number of induced colors taken over local antimagic labelings of $G$.
Let $G$ and $H$ be two disjoint graphs.
The graph $G[H]$ is obtained by the lexicographic product of $G$ and $H$.
In this paper, we obtain sufficient conditions for $\chi_{la}(G[H])\le \chi_{la}(G)\chi_{la}(H)$.
Consequently, we give
examples of $G$ and $H$ such that $\chi_{la}(G[H]) = \chi(G)\chi(H)$,
where $\chi(G)$ is the chromatic number of $G$.
We conjecture that (i) there are infinitely many graphs $G$ and $H$
such that $\chi_{la}(G[H]) = \chi_{la}(G)\chi_{la}(H) = \chi(G)\chi(H)$, and (ii) for $k\ge 1$, $\chi_{la}(G[H]) = \chi(G)\chi(H)$ if and only if $\chi(G)\chi(H) = 2\chi(H) + \lceil\frac{\chi(H)}{k}\rceil$, where $2k+1$ is the length of a shortest odd cycle in $G$.
 \\[2mm]
 {\bf Keywords:} lexicographic product; regular; local antimagic chromatic number.\\[2mm]
 {\bf 2020 Mathematics Subject Classification:} 05C78, 05C69
 \end{abstract}

\baselineskip=0.20in

\section{Introduction}
Let $G = (V,E)$ be a connected simple graph of order $p$ and size $q$.
A bijection $f : E \to \{1,2,\ldots,q\}$ is called a {\it local antimagic labeling} of $G$ if $f^+(u) \ne f^+(v)$ holds for any two adjacent vertices $u$ and $v$, where $f^+(u) = \sum_{e\in E(u)} f(e)$, and $E(u)$ is the set of edges incident to $u$.
Clearly, a local antimagic labeling induces a proper coloring of $G$.
The function $f$ is called a {\it local antimagic $t$-coloring} of $G$ if $f$ induces a proper $t$-coloring of $G$, and we say $c(f)=t$.
The {\it local antimagic chromatic number} of $G$, denoted by $\chi_{la}(G)$, is the minimum number of $c(f)$, where $f$ takes over all local antimagic labelings of $G$ \cite{Arumugam}. Interested readers may refer to~\cite{LNS, LNS-dmgt, LauShiuNg} for results related to local antimagic chromatic numbers of graphs.

Let $G$ and $H$ be two disjoint graphs.
The {\it lexicographic product} $G[H]$ of graphs $G$ and $H$ is a graph such that
its vertex set is the cartesian product $V(G) \times V(H)$,
and any two vertices $(u,u')$ and $(v,v')$ are adjcent in $G[H]$
if and only if either $uv\in E(G)$ or $u=v$ and $u'v'\in E(H)$.
In~\cite{LauShiu-lexi}, the first two authors studied the exact value of $\chi_{la}(G[O_n])$, where $O_n$ is a null graph of order $n\ge 2$.
Motivated by the above result, we investigate the sharp upper bound of $\chi_{la}(G[H])$
for any two disjoint non-null graphs $G$ and $H$ in this paper.
We present the sufficient conditions for $\chi_{la}(G[H]) \le \chi_{la}(G)\chi_{la}(H)$ holding.
Further, we conjecture that (i) there are infinitely many graphs $G$ and $H$ with  $\chi_{la}(G[H]) = \chi_{la}(G)\chi_{la}(H) = \chi(G)\chi(H)$,
where $\chi(G)$ is the chromatic number of $G$;
and (ii) for any positive integer $k$, $\chi_{la}(G[H]) = \chi(G)\chi(H)$ if and only if $\chi(G)\chi(H) = 2\chi(H) + \lceil\frac{\chi(H)}{k}\rceil$, where $2k+1$ is the length of the shortest odd cycle in $G$.
We refer to \cite{Bondy} for all undefined notation.

\section{Bounds of $\chi_{la}(G[H])$}\label{sec-lexi}

Before presenting our main results, we introduce some necessary notation and known results which will be used in this section.

Let $[a,b]=\{n\in\mathbb{Z}\;|\; a\le n\le b\}$ and $S\subseteq \mathbb{Z}$.
Let $S^-$ and $S^+$ be a decreasing sequence and an increasing sequence of $S$, respectively.

\begin{lemma}[{\cite[Lemma~2.2]{LauShiu2022-0}}]\label{lem-S}
For positive integers $q$ and $p$, let $S_p(a)=[p(a-1)+1, pa]$, $1\le a\le q$. Then,
\begin{enumerate}[(i)]
\item $\{S_p(a)\;|\; 1\le a\le q\}$ is a partition of $[1, pq]$;
\item when $a<b$, every term of $S_p(a)$ is less than that of $S_p(b)$;
\item for each $1\le i\le p$,
the sum of the $i$-th term of $S_p^+(a)$ and that of $S_p^-(b)$ is independent of
$i$, where $1\le a, b\le q$;
\item for any positive integer $k$ and each $1\leq i\leq p$, $\sum\limits_{l=1}^k \left(i\mbox{-th term of }S_p^+(a_l)\right) + \sum\limits_{l=1}^k \left(i\mbox{-th term of } S_p^-(b_l)\right)$ is independent of 
$i$, where $1\le a_l, b_l\le q$.
\end{enumerate}
\end{lemma}
 Note that the proof of Lemma \ref{lem-S} in\cite{LauShiu2022-0} shows that the sum of $i$-term of $S_p^+(a)$ and that of $S_p^-(b)$ is $p(a+b-1)+1$.
 According to the definitions of $S_p^+(a)$ and $S_p^-(a)$,
 we shall write the sequence $S_p^+(a)$ and $S_p^-(a)$ as column vectors in this paper.
 Now we are ready to present our first main result.

\begin{theorem}\label{thm-pH} Suppose $H$ admits a local antimagic $t$-coloring $f$ that satisfies the following conditions:
\begin{enumerate}[(a)]
\item for each vertex, the number of even incident edge labels equals the number of odd incident edge labels under $f$;
\item when $f^+(u) = f^+(v)$, $\deg(u) = \deg(v)$;
\item when $f^+(u)\ne f^+(v)$, $pf^+(u)-\frac{1}{2}\deg(u)(p-1) \ne pf^+(v)-\frac{1}{2}\deg(v)(p-1)$ holds for a fixed integer $p$.
\end{enumerate}
Then $\chi_{la}(pH)\le t$.
\end{theorem}

\begin{proof} Let $V(H)=\{x_1, \ldots, x_n\}$ and
$L$ be the labeling matrix of $H$ according to $f$ (for definition of labeling matrix, please see \cite{Shiu1998}). Now we define a guide matrix $\mathcal M$ by adding a `$+$' sign to all odd entries and a `$-$' sign to all even entries in $L$. The concept of guide matrix was introduced in \cite{LauShiu2022-0}.

We define $p$ matrices $L_1, \dots, L_p$ as follows.
For each $1\leq i\leq p$, the $(j,k)$-entry of $L_i$ is the $i$-th term of $S_p^+(a)$ (resp. $S_p^-(a)$) if the corresponding $(j,k)$-entry of $\mathcal M$ is $+a$ (resp. $-a$), where $1\le a\le |E(H)|$.

From the condition (a),
for each row of $L$, the number of odd entries equals that of even entries.
Thus, let $a_1, \dots, a_s$ denote the odd numerical entries of the $j$-th row of $L$ and
$b_1, \dots, b_s$ denote the even numerical entries of the $j$-th row of $L$, where $s$ is a positive integer.
Now,
\begin{align*}
r_j(L_i) & = \sum_{l=1}^s [\mbox{$i$-th term of }S_p^+(a_l)]+\sum_{l=1}^s [\mbox{$i$-th term of }S_p^-(b_l)].
\end{align*}

 By Lemma~\ref{lem-S} (iv), $r_j(L_i)$ is constant for a fixed $j$. Actually, it is $p\sum\limits_{l=1}^s (a_l+b_l)-ps+s =pr_j(L) -k(p-1)=pf^+(x_j)-\frac{1}{2}\deg(x_j)(p-1)$. By conditions (a) and (b), the diagonal block matrix \[\begin{pmatrix}L_1 & \bigstar & \cdots & \bigstar\\
\bigstar & L_2 & \cdots & \bigstar\\
\vdots & \vdots & \ddots & \vdots\\
\bigstar & \bigstar & \cdots & L_p\end{pmatrix}\]
 is a local antimagic labeling of $pH$. Thus $\chi_{la}(pH)\le t$.
\end{proof}
 It is known that $\chi_{la}(K_{1,2n})=2n+1$ and $\chi_{la}(mK_{1,2n})=2nm+1$ \cite[Corollary~3]{Baca}. Clearly, the upper bound stated in Theorem~\ref{thm-pH} is not sharp.
From Theorem \ref{thm-pH}, we obtain the following result immediately.

\begin{corollary}\label{cor-pH} If $H$ is an $r$-regular graph $(r\geq 2)$ with a local
antimagic $t$-coloring $f$ satisfying the condition (a) of Theorem \ref{thm-pH}, then $\chi_{la}(pH)\le t$ holds for any positive integer $p$.
\end{corollary}

\begin{theorem}\label{thm-G[H]}
Let $G$ be a graph of order $p$ admitting a local antimagic $\chi_{la}(G)$-coloring $g$ and $H$ be a graph of order $n$ admitting a local antimagic $\chi_{la}(H)$-coloring $h$. Suppose $h$ satisfies the following conditions:
\begin{enumerate}[(i)]
\item For each vertex, the number of even incident edge labels equals the number of odd incident edge labels under $h$;
\item when $h^+(u) = h^+(v)$, $\deg_H(u) = \deg_H(v)$;
\item when $h^+(u)\ne h^+(v)$, $ph^+(u)-\frac{1}{2}\deg_H(u)(p-1) \ne ph^+(v)-\frac{1}{2}\deg_H(v)(p-1)$.
\end{enumerate}
Moreover, $g$ satisfies the following conditions:
\begin{enumerate}[(i)]
\addtocounter{enumi}{3}
\item when $g^+(u) = g^+(v)$, $\deg_G(u) = \deg_G(v)$, and
\item when $g^+(u)\ne g^+(v)$, $g^+(u)n^3 -\frac{(n^3-n)\deg_G(u)}{2} \ne g^+(v)n^3 -\frac{(n^3-n)\deg_G(v)}{2}$.
\end{enumerate}
Then $\chi_{la}(G[H])\leq \chi_{la}(G)\chi_{la}(H)$.
\end{theorem}

\begin{proof}
Let $q(G)$ and $q(H)$ denote the sizes of $G$ and $H$ respectively.
Clearly, $G[H]$ is a graph of order $pn$ and size $pq(H)+q(G)n^2$.
Suppose that $\{u_1, \ldots, u_p\}$ and $\{x_1, \ldots, x_n\}$ are the vertex lists of $G$ and $H$ respectively.
According to these vertex lists, we define that $A_G$ and $A_H$ are the adjacency matrices of $G$ and $H$ respectively.
Thus the adjacency matrix of $G[H]$ can be expressed as
\[A_G\otimes J_n+I_p\otimes A_H,\]
where $J_n$ is an $n\times n$ matrix whose entries are all $1$,
$I_p$ is an identity matrix of order $p$,
and $A_G\otimes J_n$ is the Kronecker product of $A_G$ and $J_n$.
Note that $A_G\otimes J_n$ is the adjacency matrix of $G[O_n]$
and $I_p\otimes A_H$ is the adjacency matrix of $O_p[H]$,
where $O_n$ and $O_p$ are null graphs of orders $n$ and $p$.
Therefore, the diagonal blocks of $A_G\otimes J_n$ are zero matrices
and only the diagonal blocks of $I_p\otimes A_H$ are non-zero matrices.

Now we shall label the edges of $O_p[H]$ and $G[O_n]$ separately.
According to the definition, $O_p[H]\cong pH$.
Since $H\in \mathcal{H}$, by Theorem \ref{thm-pH},
$pH$ has a local antimagic $\chi_{la}(H)$-coloring, say $\phi$, by using integers in
$[1, pq(H)]$ such that for each vertex $(u_i, x_j)$ in $O_p[H]$,
$\phi^+(u_i, x_j)$ is independent of $i$, where $1\leq i\leq p$.
The labeling matrix of $\phi$ is denoted by $\mathscr M_1$.

Next we shall label $G[O_n]$ by integers in $[1, q(G)n^2]$.
This labeling was constructed in the proof of \cite[Theorem~2.1]{LauShiu-lexi}.
For completeness, we list the outline of the construction.

Let $M_g$ be the labeling matrix of $G$ corresponding to $g$.
Suppose $\Omega$ is a magic square of order $n$.
Let $\Omega_i=\Omega+(i-1)n^2J_n$, where $1\le i\le q(G)$
and $\psi_0$ be the labeling of $G[O_n]$ such that its labeling matrix $\mathscr M$ is defined by replacing each entry of $M_G$ with an $n\times n$ matrix as follows:
\begin{enumerate}[(1)]
\item replace $*$ by $\bigstar$ which is an $n \times n$ matrix whose entries are $*$;
\item replace $i$ by $\Omega_i$, if $i$ lies in the upper triangular part of $M_g$;
\item replace $i$ by $\Omega_i^T$, if $i$ lies in the lower triangular part of $M_g$,
where $\Omega_i^T$ is the transpose of $\Omega_i$.
\end{enumerate}

For each vertex $(u_l,x_j)\in V(G[O_n])$,
the row sum of $\mathscr M_1$ corresponding to the vertex $(u_l, x_j)$ is
$\psi_0^+(u_l, x_j)=g^+(u_l)n^3 -\frac{(n^3-n)\deg_G(u_l)}{2}$,
which is independent of $j$.
By condition (i),
$\psi_0$ is a local antimagic labeling of $G[O_n]$.
According to condition (v),
there are at most $\chi(G)$ distinct row sums of $\mathscr M$.
Let $\mathscr M_2$ be the matrix obtained from $\mathscr M$ by adding all numerical entries with $pq(H)$
and $\psi$ be the corresponding labeling.
Then, $\psi^+(u_l, x_j)=\psi_0^+(u_l, x_j)+npq(H)$,
which is independent of $j$.

Therefore, $\mathscr M_1+\mathscr M_2$ is a labeling matrix that corresponds to a local antimagic labeling of $G[H]$, where $*$ is treated as 0.
Hence $\chi_{la}(G[H])\le \chi_{la}(G)\chi_{la}(H)$.
\end{proof}

The following is an example of Theorem \ref{thm-G[H]}.

\begin{example}\label{eg-OUC4[OUC3]}
Let $G$ be the one point union of two $4$-cycles and $H$ be the one point union of two $3$-cycles. Figure \ref{fig:Ex} show the local antimagic $3$-colorings of $G$ and $H$.

\begin{figure}[H]
    \centering
    \subfloat[Graph $G$]{{\includegraphics[width=5cm]{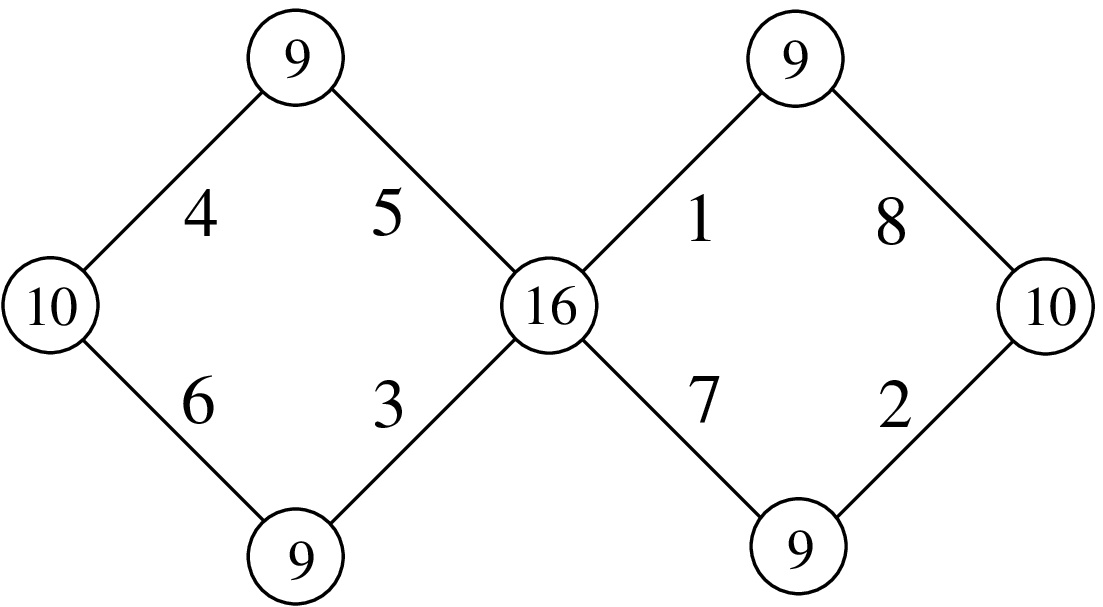} }}
    \qquad
    \subfloat[Graph $H$]{{\includegraphics[width=5cm]{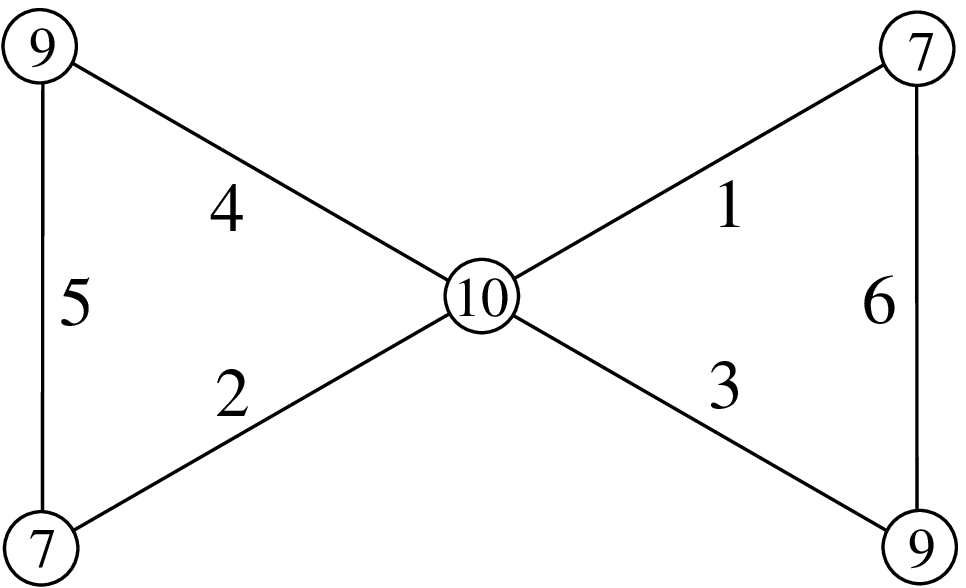} }}
    \caption{Local antimagic $3$-colorings of graphs $G$ and $H$}
    \label{fig:Ex}
\end{figure}

Note that $\chi_{la}(G)=\chi_{la}(H)=3$.
It is easy to check that the above local antimagic $3$-colorings of $G$ and $H$ satisfy the conditions of Theorem~\ref{thm-G[H]} respectively.
Then, the labeling matrices of $G$ and $H$ are shown below:
\[M_g=\left(\begin{array}{ccccccc}
* & * & * & * & 8 &  * & 1\\
* & * & * & * & 2 &  * & 7\\
* & * & * & * & * &  6 & 3\\
* & * & * & * & * &  4 & 5\\
8 & 2 & * & * & * & * & *\\
* & * & 6 & 4 & * & * & *\\
1 & 7 & 3 & 5 & * & * & *
\end{array}\right),\qquad
M_h=\left(\begin{array}{ccccc}
* & * & 6 &  * & 1\\
* & * & * &  5 & 2\\
6 & * & * &  * & 3\\
* & 5 & * &  * & 4\\
1 & 2 & 3 & 4 & *
\end{array}\right).\]
Let
{\small
\begin{alignat*}{2}
L_1=\left(\begin{array}{ccccc}
* & * & 42 &  * & 1\\
* & * & * &  29 & 14\\
42 & * & * &  * & 15\\
* & 29 & * &  * & 28\\
1 & 14 & 15 & 28 & *
\end{array}\right), & \quad
L_2=\left(\begin{array}{ccccc}
* & * & 41 &  * & 2\\
* & * & * &  30 & 13\\
41 & * & * &  * & 16\\
* & 30 & * &  * & 27\\
2 & 13 & 16 & 27 & *
\end{array}\right), &\quad
L_3=\left(\begin{array}{ccccc}
* & * & 40 &  * & 3\\
* & * & * &  31 & 12\\
40 & * & * &  * & 17\\
* & 31 & * &  * & 26\\
3 & 12 & 17 & 26 & *
\end{array}\right),\\
L_4=\left(\begin{array}{ccccc}
* & * & 39 &  * & 4\\
* & * & * &  32 & 11\\
39 & * & * &  * & 18\\
* & 32 & * &  * & 25\\
4 & 11 & 18 & 25 & *
\end{array}\right), &\quad
L_5=\left(\begin{array}{ccccc}
* & * & 38 &  * & 5\\
* & * & * &  33 & 10\\
38 & * & * &  * & 19\\
* & 33 & * &  * & 24\\
5 & 10 & 19 & 24 & *
\end{array}\right), & \quad
L_6=\left(\begin{array}{ccccc}
* & * & 37 &  * & 6\\
* & * & * &  34 & 9\\
37 & * & * &  * & 20\\
* & 34 & * &  * & 23\\
6 & 9 & 20 & 23 & *
\end{array}\right),\\
L_7=\left(\begin{array}{ccccc}
* & * & 36 &  * & 7\\
* & * & * &  35 & 8\\
36 & * & * &  * & 21\\
* & 35 & * &  * & 22\\
7 & 8 & 21 & 22 & *
\end{array}\right). &&
\end{alignat*}
}

Obviously, for each $1\leq i\leq 7$, the row sums of $L_i$ are 43, 43, 57, 57, 58 respectively.
Let $\Omega$ be a magic square of order 5 with row sum $65$ and $\Omega_i=\Omega +25(i-1)J_5$,
where $1\le i\le 8$.
For each $1\leq i\leq 8$, let
$\Psi_i=\Omega_i+42J_5$. Then, the labeling matrix of $G[H]$ is

\[\begin{pmatrix}
L_1 & \bigstar & \bigstar & \bigstar & \Psi_8 & \bigstar & \Psi_1\\
\bigstar & L_2 & \bigstar & \bigstar & \Psi_2 & \bigstar & \Psi_7\\
\bigstar & \bigstar & L_3 & \bigstar & \bigstar & \Psi_6 & \Psi_3\\
\bigstar & \bigstar & \bigstar & L_4 & \bigstar & \Psi_4 & \Psi_5\\
\Psi_8^T & \Psi_2^T &  \bigstar & \bigstar & L_5 & \bigstar & \bigstar\\
\bigstar & \bigstar & \bigstar & \Psi_6^T & \Psi_4^T & L_6 & \bigstar\\
\Psi_1^T & \Psi_7^T & \Psi_3^T & \Psi_5^T & \bigstar & \bigstar & L_7
\end{pmatrix}\]
By calculating the row sums of the above matrix, we obtain that
the distinct row sums are 1468, 1482, 1483, 1593, 1607, 1608, 2643, 2657, 2658.
Thus, $\chi_{la}(G[H])\le 9$. 
\end{example}

In \cite{Cizek+K}, N. \v{C}i\v{z}ek and S. Klav\v{z}ar gave the lower bound of chromatic number of the lexicographic product as follows.

\begin{corollary}[{\cite[Corollary~3]{Cizek+K}}]\label{cor-LB1} Let $G$ be a nonbipartite graph. Then for any graph $H$, $\chi(G[H])\ge 2\chi(H) + \big\lceil\frac{\chi(H)}{k}\big\rceil$, where $k\geq 1$ and $2k + 1$ is the length of a shortest odd cycle in $G$.
\end{corollary}

Combining Theorem \ref{thm-G[H]} and Corollary \ref{cor-LB1}, we obtain the following results.

\begin{corollary}\label{cor-ULbounds}
Suppose $G$ and $H$ are graphs satisfying the conditions listed in Theorem~\ref{thm-G[H]}.
If the length of a shortest odd cycle in $G$ is $2k+1$, then
 $2\chi(H) + \lceil\frac{\chi(H)}{k}\rceil \le \chi(G[H]) \le \chi_{la}(G[H]) \le \chi_{la}(G)\chi_{la}(H)$.
In particular, if $C_3$ is a subgraph of $G$, then $3\chi(H)  \le \chi_{la}(G[H])\le \chi_{la}(G)\chi_{la}(H)$.
\end{corollary}

\begin{proof} $\chi(G[H]) \le \chi_{la}(G[H])$ is trivial. The lower bound follows from Corollary~\ref{cor-LB1} and the upper bound follows from Theorem~\ref{thm-G[H]}.
\end{proof}

\begin{corollary}\label{cor-ULbounds-reg} Let $G$ and $H$ be regular graphs and $H$ admit a local antimagic $\chi_{la}(H)$-coloring $h$.
Suppose for each vertex of $H$, the number of even incident edge labels equals the number of odd incident edge labels under $h$.
If the length of a shortest odd cycle in $G$ is $2k+1$, then $2\chi(H) + \lceil\frac{\chi(H)}{k}\rceil \le \chi(G[H]) \le \chi_{la}(G[H]) \le \chi_{la}(G)\chi_{la}(H)$.
In particular, if $C_3$ is a subgraph of $G$,
then $3\chi(H)  \le \chi_{la}(G[H])\le \chi_{la}(G)\chi_{la}(H)$.
\end{corollary}

By applying Corollary \ref{cor-ULbounds-reg}, we can obtain $\chi_{la}(G[H])$ for some graphs $G$ and $H$. An example is shown in Example \ref{Ex-2}.

\begin{example}\label{Ex-2}
Let $G=C_3\times K_2$ and $H$ be the octahedral graph. Figure \ref{fig:Ex-2} presents their local antimagic $3$-colorings which are shown in \cite{LauShiu2022-1}.

\begin{figure}[H]
    \centering
    \subfloat[Graph $G=C_3\times K_2$]{{\includegraphics[width=4cm]{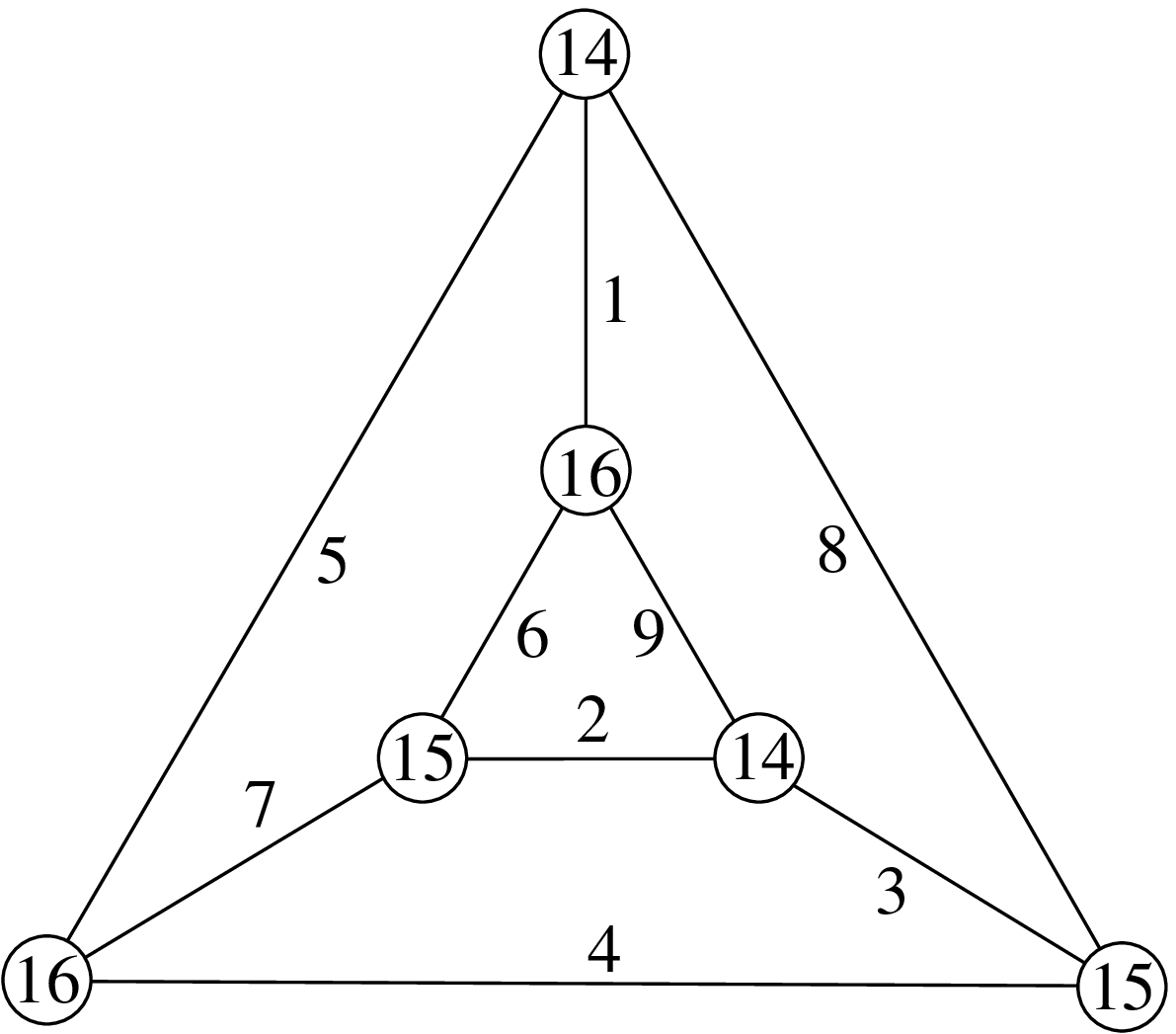} }}
    \qquad
    \subfloat[Graph $H$]{{\includegraphics[width=3cm]{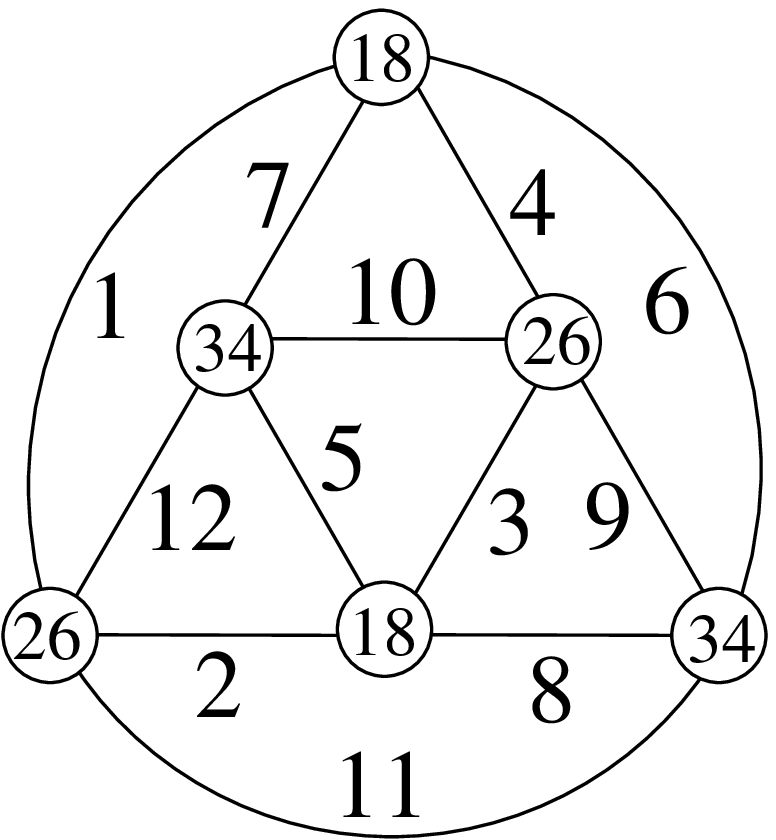} }}
    \caption{Local antimagic $3$-colorings of graphs $G$ and $H$}
    \label{fig:Ex-2}
\end{figure}

It is easy to verify that $G$ and $H$ satisfy the conditions of Corollary~\ref{cor-ULbounds-reg},
which implies that $3\chi(H) \le \chi_{la}(G[H]) \le \chi_{la}(G)\chi_{la}(H)$. Since $\chi_{la}(G) = \chi_{la}(H) =3$,  $\chi_{la}(G[H]) = 9$. 
\end{example}

%

%

In~\cite[Theorem 3.3]{LNS-dmgt}, the first two authors proved that $\chi_{la}(C_{2m} \vee O_{2n}) = 3$ for $m\ge 2$, $n\ge 1$, where $C_{2m} \vee O_{2n}$ is the join of graphs $C_{2m}$ and $O_{2n}$.
In the following, we give another local antimagic 3-coloring of $C_{2m} \vee O_{2n}$ that satisfies  the conditions (i) and (ii)  of Theorem~\ref{thm-G[H]}.

\begin{theorem}\label{thm-C2mVO2n} For $m\ge 2$ and $n\ge 1$, there is a local antimagic $3$-coloring of $C_{2m}\vee O_{2n}$ satisfying conditions (i) and (ii) of Theorem~\ref{thm-G[H]}.
\end{theorem}

\begin{proof} Let $V(C_{2m})=\{u_i\;|\;1\le i\le 2m\}$ and $V(O_{2n})=\{v_j\;|\; 1\le j\le 2n\}$.
We separate $C_{2m}\vee O_{2n}$ into two edge-disjoint graphs, $C_{2m}$ and $O_{2m}\vee O_{2n}$, where $V(O_{2m})=V(C_{2m})$.

 Firstly, define a labeling $f$ for $C_{2m}$. Let $f: V(C_{2m})\to [1, 2m]$ such that $f(u_iu_{i+1})=i$, where $1\le i\le 2m$ and $u_{2m+1}=u_1$.
Thus, $f^+(u_1)=2m+1$, $f^+(u_{i})=2i-1$ for $2\le i\le 2m$.

 Next, we define a labeling $g$ for $O_{2m}\vee O_{2n}\cong K_{2m, 2n}$. The labeling matrix of $g$ is $\begin{pmatrix}\bigstar & B\\
B^T & \bigstar\end{pmatrix}$ under the vertex lists $V(O_{2m})=\{u_1, u_3,\dots, u_{2m-1}, u_2, \dots, u_{2m}\}$ and $V(O_{2n})=\{v_1, v_2,\dots, v_{2n}\}$. So we only need to fill the integers in $[2m+1, 2m+4mn]$ into the matrix $B$.

 Let $\mathcal M$ be a guide matrix as follows:

{\fontsize{6}{9}\selectfont
\[\left(\begin{array}{cc|cc|ccccccc} -2 & -3 & +\boxed{2} & -\boxed{2n-1} & +\boxed{3} & -(2n-1) & +\boxed{5} & -(2n-3) & \cdots & +\boxed{2n-3}& -5\\
-\boxed{2n+1} & -\boxed{2n} & -(2n+1) & +4 & -(2n) & +\boxed{4} & -(2n-2) & +\boxed{6} & \cdots & -6 & +\boxed{2n-2}\end{array}\right).\]}

 We replace each entry of $\mathcal M$ to a column vector according to the rules:
\begin{enumerate}[(1)]
\item replace $-a$ to $2S_m^-(a)-J_{m,1}$; replace $+a$ to $2S_m^+(a)-J_{m,1}$,
where $J_{m, 1}$ is an $m\times 1$ matrix with all entries 1;
\item replace $-\boxed{a}$ to $2S_m^-(a)$; replace $+\boxed{a}$ to $2S_m^+(a)$.
\end{enumerate}
Let $\begin{pmatrix} B_1 \\B_2\end{pmatrix}$ be the resulting matrix, where $B_1$ and $B_2$ are $m\times 2n$ matrices.
The row sums of $B_1$ in column matrix is
\begin{align*}& \quad (2S_m^-(2)-J_{m,1})+(2S_m^-(3)-J_{m,1})+2S_m^+(2)+2S_m^-(2n-1) +\sum_{i=1}^{n-2} 2S_m^+(2i+1) + \sum_{j=2}^{n-1}[2S_m^-(2j+1)-J_{m,1}]\\
& = 4S_m^-(2n-1)+2[S_m^-(2)+S_m^+(2)]+2\sum_{i=1}^{n-2}[S_m^+(2i+1)+S_m^-(2i+1)]-nJ_{m,1}\\
& = 4S_m^-(2n-1) +2(3m+1)J_{m,1} +2\sum_{i=1}^{n-2} [m(4i+1)+1]J_{m,1} -nJ_{m,1} \\
& = 4S_m^-(2n-1) +[4mn^2-10mn+10m+n-2]J_{m,1}=A_1.\end{align*}

Clearly, the entries of the column matrix $A_1$ form a decreasing sequence with common difference 4.
Now the first column of $B_1$ is the vector $2S_m^-(2)-J_{m,1}$. We shift each entry of this vector downward to 1 and move the last entry of this vector to the top,
i.e., add the entries by 2 except the $(1,1)$-entry and subtract the $(1,1)$-entry by $2(m-1)$. Let this new matrix be $B_1'$.
Now, the first column of $B_1'$ has entries $2m+1,4m-1,4m-3,\ldots,2m+3$ so that the second entry up to the last entry of the first column of $B_1'$ form a decreasing sequence with common difference 2 and the difference between the first entry and second entry is $2-2m$.

Similarly the row sums of $B_2$ in column matrix is
\vskip-5mm

\begin{align*}& \quad 4S_m^-(2n+1)+4S_m^-(2n)+4S_m^+(4) +2\sum_{i=3}^{n-1}[S_m^+(2i)+S_m^-(2i)]-nJ_{m,1}\\
& = 4S_m^-(2n+1)+4(m(2n+3)+1)J_{m,1}+ [2m(n-3)(2n+3)+2(n-3)]J_{m,1}-nJ_{m,1}\\
& = 4S_m^-(2n+1) +[4mn^2+2mn-6m+n-2]J_{m,1}=A_2.
\end{align*}

It is clear that the entries of the column matrix $A_2$ form a decreasing sequence with common difference 4.

Combining the labelings $f$ and $g$,
we have a labeling $\phi$ for the whole graph $C_{2m}\vee O_{2n}$.
One may check that $\phi^+(u_{2j-1})=f^+(u_{2j-1})+r_j(B_1')=4mn^2-2mn+6m+n+1$ for each $1\le j\le m$; and
$\phi^+(u_{2i})=f^+(u_{2i})+r_i(B_2)=4mn^2+10mn-2m+n+1$ for each $1\le i\le m$.
Hence $\phi^+(u_{2i})>\phi^+(u_{2j-1})$ for $1\le i,j\le m$.

Clearly, the column sum of $\begin{pmatrix} B_1' \\B_2\end{pmatrix}$ is $(4mn+4m+1)m$. So $\phi^+(v_l)=(4mn+4m+1)m$.
\begin{align}\phi^+(u_{2j-1})-\phi^+(v_l) & =4mn^2-4m^2n-4m^2-2mn+5m+n+1\nonumber\\
& =4mn(n-m-1)-4m^2+2mn+5m+n+1. \label{eq-1}
\end{align}
If $n\ge m+2$, then $\phi^+(u_{2i})-\phi^+(v_l)>\phi^+(u_{2j-1})-\phi^+(v_l)\ge 4mn-4m^2+2mn+5m+n+1>0$.
\begin{align}\phi^+(v_l)-\phi^+(u_{2i}) & =  4m^2n-4mn^2+4m^2-10mn+3m-n-1\nonumber\\
& = 4mn(m-n-2) +4m^2-2mn+3m-n-1. \label{eq-2}\end{align}
If $m\ge n+2$, then $\phi^+(v_l)-\phi^+(u_{2j-1}) >\phi^+(v_l)-\phi^+(u_{2i}) >0$.
\begin{enumerate}[1)]
\item If $n=m+1$, then $\phi^+(u_{2j-1})-\phi^+(v_l)=-2m^2+8m+2\ne 0$ (since the discriminant is not a prefect square) and $\phi^+(u_{2i})-\phi^+(v_l)=10m^2+12m+2>0$
\item If $n=m$, then $\phi^+(u_{2j-1})-\phi^+(v_l)=-6m^2+6m+1<0$ and $\phi^+(u_{2i})-\phi^+(v_l)=6m^2-2m+1>0$.
\item If $n=m-1$, then $\phi^+(u_{2j-1})-\phi^+(v_l)=-10m^2+12m<0$, but $\phi^+(u_{2i})-\phi^+(v_l)=2m^2-8m\ne 0$ when $m\ne 4$.
So, for $n=m-1=3$, we have to find another labeling for $C_8\vee O_6$.

Label the edges of $C_8$ by 1, 8, 3, 2, 5, 4, 7, 6 in natural order. Let this labeling be $f$. So the induced vertex of $u_1, u_2, \dots, u_8$ are 7, 9, 11, 5, 7, 9 ,11, 13.

We start from a $8\times 6$ magic rectangle $\Omega$ (shown below).
 Add each entry by 8 and swap some entries within the same column (indicated in italic).  We have
\[\Omega=\begin{pmatrix}
1 & 44 & 9 & 36 & 29 & 28\\
2 & 43 & 10 & 35 & 30 & 27\\
3 & 42 & 11 & 34 & 31 & 26\\
4 & 41 & 12 & 33 & 32 & 25\\
45 & 8 & 37 & 16 & 17 & 24\\
46 & 7 & 38 & 15 & 18 & 23\\
47 & 6 & 39 & 14 & 19 & 22\\
48 & 5 & 40 & 13 & 20 & 21
\end{pmatrix}\longrightarrow \begin{array}{c}7\\11\\7\\11\\9\\9\\13\\5\end{array}\left(\begin{array}{cccccc}
9 & 52 & 17 & 44 & 37 & 36\\
10 & 51 & 18 & 43 & 38 & \it{31}\\
11 & 50 & 19 & 42 & 39 & 34\\
12 & 49 & 20 & 41 & 40 & \it{29}\\
53 & 16 & 45 & 24 & \it{27} & 32\\
54 & 15 & 46 & 23 & 26 & \it{33}\\
55 & 14 & 47 & 22 & \it{25} & 30\\
56 & 13 & 48 & 21 & 28 & \it{35}
\end{array}\right)\begin{array}{c}202\\202\\202\\202\\206\\206\\206\\206\end{array}\]
This matrix forms a labeling matrix of a labeling $g$ of $K_{8,6}$ under the vertex list $\{u_1, u_3, u_5, u_7, u_2, u_6,\break u_8, u_4\}$ of $C_8$,
The column in front of the matrix is the corresponding induced vertex labels under $f$ on $C_8$, and the column behind of the matrix is the induced vertex labels of the labeling $\phi$ for $C_8\vee O_6$.  Thus $\phi^+(u_{2i-1})=202$, $\phi^+(u_{2i})=206$ and $\phi^+(v_j)=260$ for $1\le i\le 4$ and $1\le j\le 6$.
\end{enumerate}
Clearly all labels are used. So $\phi$ is a local antimagic $3$-coloring for $C_{2m}\vee O_{2n}$. Moreover, the number of even incident edge labels equals the number of odd incident edge labels for each vertex. Hence $\phi$ satisfies condition (i) and (ii) of Theorem~\ref{thm-G[H]}
\end{proof}

\begin{corollary}\label{cor-C3K2[CmOn]} If $G=C_3\times K_2$ and $H=C_{2m} \vee O_{2n}$, $m\ge 2$, $n\ge 1$, then $\chi_{la}(G[H]) = 9$. \end{corollary}

\begin{proof} Keep all notation defined in the proof of Theorem~\ref{thm-C2mVO2n}. Now $\deg_H(u_i)=2n+2$, $\deg_H(v_l)=2m$ and $p=6$. By Theorems~\ref{thm-G[H]} and \ref{thm-C2mVO2n}, it suffices to check condition (iii) of Theorem~\ref{thm-G[H]}, i.e., $6[\phi^+(u_{1})-\phi^+(v_{1})]-5(n+1-m)\ne 0$ and $6[\phi^+(v_1)-\phi^+(u_2)] -5(m-n-1)\ne 0$.

 By \eqref{eq-1}, we have
\begin{align}& \quad 6[4mn(n-m-1)-4m^2+2mn+5m+n+1]-5(n+1-m)\nonumber\\& =-24m^2+24mn^2-24m^2n-12mn+35m+n+1\nonumber\\
& = 24mn(n-m)-24m(m-1)-12mn+11m+n+1.\label{eq-3}\end{align}

 Clearly $\eqref{eq-3}<0$ for $n\ge m$. When $n\ge m+2$, $\eqref{eq-3}\ge 36mn-24m^2+35m+n+1>0$. When $n=m+1$, $\eqref{eq-3}=12mn-24m^2+35m+n+1=-12m^2+48m+2\ne 0$ since the discriminant is 2400 which is not a prefect square.

 By \eqref{eq-2}, we have
\begin{align}& \quad
6[4mn(m-n-2) +4m^2-2mn+3m-n-1] -5(m-n-1)\nonumber\\ & = 24m^2+24m^2n-24mn^2-60mn+13m-n-1\nonumber\\
& = 24mn(m-n-2)+12m(m-n)+12m^2+13m-n-1.\label{eq-4}\end{align}
Clearly $\eqref{eq-4}>0$ for $m\ge n+2$.  When $m\le n$, $\eqref{eq-4}\le -48mn+12m^2+13m-n-1 =12m(m-4n+1)+m-n-1<0$. When $m=n+1$, then $H$ is regular so condition (iii) holds. The proof is complete.   \end{proof}

\begin{example}
Let $V(C_6)=\{u_1, u_3, u_5, u_2, u_4, u_6\}$ and $V(O_8)=\{v_j\;|\; 1\le j\le 8\}$ be the vertex lists of $C_6$ and $O_8$. According to the proof of Theorem~\ref{thm-G[H]} we label the edges of $C_6$ by 1 to 6 in natural order. So the induced vertex labels are 7, 3, 5, 7, 9, 11.
Then

\[\begin{array}{c}7\\5\\9\\\hline 3\\7\\11\end{array}\left(\begin{array}{cccc|cccc}
7 & 17 & 8 & 42 & 14 & 41 & 26 & 29\\
11 & 15 & 10 & 40 & 16 & 39 & 28 & 27\\
9 & 13 & 12 & 38 & 18 & 37 & 30 & 25\\\hline
54 & 48 & 53 & 19 & 47 & 20 & 35 & 32\\
52 & 46 & 51 & 21 & 45 & 22 & 33 & 34\\
50 & 44 & 49 & 23 & 43 & 24 & 31 & 36
\end{array}\right)
\begin{array}{c}191\\191\\191\\\hline 311\\311\\311\end{array}\]
The column in front of the matrix is the corresponding induced vertex labels under $f$on $C_6$, and the column behind of the matrix is the induced vertex labels of the labeling $\phi$ for $C_6\vee O_8$. One may check that the column sum of the matrix is 183, which is $\phi^+(v_j)$ for all $j$. 
\end{example}

 Corollary~\ref{cor-C3K2[CmOn]} shows that there exist infinitely many graphs $H$ such that $\chi_{la}(G[H]) = \chi_{la}(G)\chi_{la}(H) = \chi(G)\chi(H)$. We end this note with the following conjectures.

\begin{conjecture} There exist infinitely many graphs $G$ and $H$ respectively such that  $\chi_{la}(G[H]) = \chi_{la}(G)\chi_{la}(H)  = \chi(G)\chi(H)$.  \end{conjecture}

\begin{conjecture} For $k\ge 1$ and graph $G$ and $H$, $\chi_{la}(G[H]) = \chi(G)\chi(H)$ if and only if $\chi(G)\chi(H) = 2\chi(H) + \lceil\frac{\chi(H)}{k}\rceil$, where $2k+1$ is the length of a shortest odd cycle in $G$. \end{conjecture}

\end{document}